\documentclass[10pt,leqno]{amsart}

\usepackage{amscd}
\usepackage{color}
\usepackage{amssymb}
\usepackage{amsfonts}
\usepackage{latexsym}
\usepackage{verbatim}
\usepackage{tikz}
\usetikzlibrary{matrix,arrows}
\usepackage[all]{xy}

\theoremstyle{plain}
\newtheorem{theorem}{Theorem}[section]
\newtheorem{definition}[theorem]{Definition}
\newtheorem{lemma}[theorem]{Lemma}
\newtheorem{prop}[theorem]{Proposition}
\newtheorem{cor}[theorem]{Corollary}
\newtheorem{rem}[theorem]{Remark}
\newtheorem{ex}[theorem]{Example}

\renewcommand{\b}{\begin{equation}}
\newcommand{\e}{\end{equation}}
\newcommand{\g}{\mathfrak{g}}
\newcommand\C{{\mathbb C}}

\newcommand{\del}{\partial}
\newcommand{\delbar}{\overline\partial}
\begin{document}

\title[HKT manifolds: Hodge Theory, formality and balanced metrics]{HKT manifolds: Hodge Theory, formality and balanced metrics}

\author{Giovanni Gentili}
\address[Giovanni Gentili]{Dipartimento di Matematica ``G. Peano'' \\
Universit\`{a} degli studi di Torino \\
Via Carlo Alberto 10\\
10123 Torino, Italy}
\email{giovanni.gentili@unito.it}

\author{Nicoletta Tardini}
\address[Nicoletta Tardini]{Dipartimento di Scienze Matematiche, Fisiche e Informatiche\\
Unit\`a di Matematica e Informatica\\
Universit\`a degli Studi di Parma\\
Parco Area delle Scienze 53/A\\
43124 Parma, Italy}
\email{nicoletta.tardini@unipr.it}


\begin{abstract} 
Let $(M,I,J,K,\Omega)$ be a compact HKT manifold and denote with $\partial$ the conjugate Dolbeault operator with respect to $I$, $\partial_J:=J^{-1}\overline\partial J$, $\partial^\Lambda:=[\partial,\Lambda]$ where $\Lambda$ is the adjoint of $L:=\Omega\wedge-$. Under suitable assumptions, we study Hodge theory for the complexes $(A^{\bullet,0},\partial,\partial_J)$ and $(A^{\bullet,0},\partial,\partial^\Lambda)$ showing a similar behavior to K\"ahler manifolds. In particular, several relations among the Laplacians, the spaces of harmonic forms and the associated cohomology groups, together with Hard Lefschetz properties, are proved. Moreover, we show that for a compact HKT $\mathrm{SL}(n,\mathbb{H})$-manifold the differential graded algebra $(A^{\bullet,0},\partial)$ is formal and this will lead to an obstruction for the existence of an HKT $\mathrm{SL}(n,\mathbb{H})$-structure $(I,J,K,\Omega)$ on a compact complex manifold $(M,I)$. Finally, balanced HKT structures on solvmanifolds are studied.
\end{abstract}
\maketitle

\section{Introduction}
Let $ (M,I,J,K) $ be a \emph{hypercomplex manifold}, i.e. a smooth manifold $ M $ equipped with three complex structures $ I,J,K $ that anticommute with each other, and such that $ IJ=K $. On a hypercomplex manifold there is a distinguished connection $ \nabla $, called the Obata connection \cite{Obata (1956)}, that is torsion-free and preserves the hypercomplex structure, in the sense that
\[
\nabla I=0\,,\qquad \nabla J=0\,,\qquad \nabla K=0\,.
\]

A Riemannian metric that is Hermitian with respect to all three complex structures $ I,J,K $ is said to be \emph{hyperhermitian}, and, accordingly, $ (M,I,J,K,g) $ is called an hyperhermitian manifold.

On a hyperhermitian manifold the $ 2 $-form
\[
\Omega:=\frac{g(J\cdot,\cdot)+ig(K\cdot,\cdot)}{2} 
\]
is non-degenerate of type $ (2,0) $ with respect to $ I $ and completely determines the hyperhermitian metric via the relation $ \Omega(Z,J\bar Z)=g(Z,\bar Z) $ for every $ Z\in T^{1,0}M $. Furthermore there is a bijective correspondence between hyperhermitian metrics and q-positive q-real $ (2,0) $-forms, where a $ (2,0) $-form $ \Omega $ is called q-real if $ J\Omega =\bar \Omega $ and it is called q-positive if additionally satisfies $ \Omega(Z,J\bar Z)>0 $ for every non-zero $ Z\in T^{1,0}M $.

A hyperhermitian manifold is \emph{HKT}, namely \emph{hyperk\"ahler with torsion}, if the associated $ (2,0) $-form satisfies
\[
\partial \Omega=0\,,
\]
where $ \partial  $ is the Dolbeault operator with respect to $ I $. The HKT condition is at many levels the hypercomplex analogue of the K\"ahler condition (see e.g. \cite{Alesker-Verbitsky (2006),Banos-Swann,grantcharov-lejmi-verbitsky,Grantcharov-Poon,lejmi-weber,verbitsky-hodge,Verbitsky (2009)}). The main purpose of this work is twofold: on one hand we explore the analogies with the K\"ahler setting from a cohomological point of view, as done in \cite{grantcharov-lejmi-verbitsky,lejmi-weber,verbitsky-hodge}, on the other hand we extend some results proved on hypercomplex nilmanifolds in \cite{barberis-dotti-verbitsky} to hypercomplex solvmanifolds.

In \cite{Verbitsky (2007)} Verbitsky proved that a compact HKT manifold has trivial canonical bundle if and only if the holonomy group of the Obata connection is contained in $ \mathrm{SL}(n,\mathbb{H}) $. Manifolds with the latter property are called  {\em $ \mathrm{SL}(n,\mathbb{H}) $-manifolds}. If one removes the HKT hypothesis, then it is still true that an $ \mathrm{SL}(n,\mathbb{H}) $-manifold has holomorphically trivial canonical bundle. The converse has been recently disproved by Andrada-Tolcachier \cite{Andrada-Tolcachier}, however it is, for instance, true for hypercomplex nilmanifolds \cite{barberis-dotti-verbitsky}. We prove here that on hypercomplex solvmanifolds the $ \mathrm{SL}(n,\mathbb{H}) $ condition is equivalent to the existence of an invariant holomorphic trivialization of the canonical bundle (Theorem \ref{thm:solv_SLnH}).\\

Another very relevant property in this context is the balanced condition. Indeed, it turns out \cite{Verbitsky (2009)} that for a compact HKT manifold $ (M,I,J,K,g) $ one has that $ (M,I,g) $ is balanced if and only if $ \partial \bar \Omega^n=0 $. In particular, if a compact HKT manifold is balanced it is necessarily $ \mathrm{SL}(n,\mathbb{H}) $. A very interesting conjecture posed by Alesker and Verbitsky predicts that this condition is also sufficient,  in the sense that an HKT $ \mathrm{SL}(n,\mathbb{H}) $-manifold admits a balanced HKT metric (which does not necessarily coincide with the initial one). Again, evidence for this conjecture is provided by nilmanifolds  with  invariant hypercomplex structure \cite{barberis-dotti-verbitsky} and we shall extend this fact to solvmanifolds with  invariant hypercomplex structure (Theorem \ref{Teor:solv}).

As explained in \cite{Verbitsky (2009)}, one way to approach this problem is by studying the quaternionic Calabi-Yau conjecture proposed by Alesker and Verbitsky in \cite{Alesker-Verbitsky (2010)}. The solvability of the relative equation would allow to prescribe the complex volume of the HKT manifold. More precisely, for any q-positive $ (2n,0) $-form $ \Theta $, one could find a new HKT form $ \Omega' $, compatible with the given hypercomplex structure, satisfying
\[
(\Omega')^n=\Theta\,.
\]
Choosing $ \Theta $ to be holomorphic would imply that $ \Omega' $ is balanced HKT. For further information on the (still unsolved in general) quaternionic Calabi-Yau conjecture see e.g. \cite{Alesker-Shelukhin (2017),BGV,DinewSroka,GentiliVezzoni,GV,GZ,Sroka,Z} and references therein.\\
Other sufficient conditions on solvmanifolds with an invariant HKT structure to have a balanced metric are given in Theorems \ref{thm:abelian-then-balanced}, \ref{thm:trivialcan-then-balanced}.
\medskip

On a hypercomplex manifold $ (M,I,J,K) $ endowed with a HKT structure $\Omega$  there are three important differential operators. Denote with $ A^{p,q}(M)=A^{p,q}_I(M) $ the space of $ (p,q) $-forms with respect to $ I $. Then we are interested in the usual (conjugate) Dolbeault operator $ \partial \colon A^{p,q}(M) \to A^{p+1,q}(M) $ and the twisted differential operator $ \partial_J \colon A^{p,q}(M) \to A^{p+1,q}(M) $ defined as $ \partial_J:=J^{-1}\bar \partial J $. Because of the integrability of $ I,J,K $, these operators anticommute and both square to zero. Therefore we obtain a cochain complex $ (A^{\bullet,q}, \partial,\partial_J) $ for every fixed $ q $. This differs from the complex case, as here we obtain a single complex, while in the complex setting $ \partial $ and $ \bar \partial $ give rise to a double complex. Here, we restrict to study the case $ q=0 $.\\
 Moreover, the existence of $\Omega$ leads to the definition of the operator $\del^\Lambda:=[\del,\Lambda]$ where $\Lambda$  is the adjoint of $L:=\Omega\wedge-$. Therefore we obtain a cochain complex $ (A^{\bullet,q}, \partial,\partial^\Lambda) $ for every fixed $ q $, and again we will restrict to $q=0$.
 Hence one can study cohomology groups and Hodge theory from a ``complex point of view'' on $ (A^{\bullet,0}, \partial,\partial_J) $ or from a ``symplectic point of view'' on $ (A^{\bullet,0}, \partial,\partial^\Lambda) $.\\
 Concerning the latter case, one can contextualize everything in the more general setting of Lefschetz spaces and with this approach one can then prove that all the natural Laplacian operators that arise in Section \ref{Sec:Lef} and at the beginning of Section \ref{Sec:HKT} are related.
From the complex point of view, naturally, the quaternionic Dolbeault, Bott-Chern and Aeppli cohomology groups can be defined:

\[
H^{p,0}_\partial(M):=\frac{\mathrm{Ker}(\partial\vert_{A^{p,0}(M)})}{\partial A^{p-1,0}(M)}\,, \qquad H^{p,0}_{\partial_J}(M):=\frac{\mathrm{Ker}(\partial_J\vert_{A^{p,0}(M)})}{\partial_J A^{p-1,0}(M)}\,,
\]
\[
H^{p,0}_{\mathrm{BC}}(M):=\frac{\mathrm{Ker}(\partial\vert_{A^{p,0}(M)})\cap \mathrm{Ker}(\partial_J\vert_{A^{p,0}(M)})}{\partial\partial_J A^{p-2,0}(M)}\,,
\]
\[
H^{p,0}_{\mathrm{A}}(M):=\frac{\mathrm{Ker}(\partial\partial_J\vert_{A^{p,0}(M)})}{\partial A^{p-1,0}(M)+ \partial_J A^{p-1,0}(M)}\,,
\]
when $ M $ is compact all these groups are finite-dimensional \cite{grantcharov-lejmi-verbitsky}, indeed, as usual, once fixed an hyperhermitian metric, one can show that each of these cohomology group is isomorphic to the kernel of the following Laplacians acting on $ (p,0) $-forms
\[
\Delta_\partial:=\partial \partial^*+\partial^*\partial\,, \qquad \Delta_{\partial_J}:=\partial_J \partial^*_J+\partial^*_J\partial_J\,,
\]
\[
\Delta_{\mathrm{BC}}:=\partial^* \partial +\partial^*_J\partial_J+\partial \partial_J \partial_J^* \partial^*+\partial^*_J\partial^*\partial \partial_J+\partial^*_J \partial \partial^* \partial_J+\partial^*\partial_J \partial_J^* \partial\,,
\]
\[
\Delta_{\mathrm{A}}:=\partial \partial^* +\partial_J\partial_J^*+\partial \partial_J \partial_J^* \partial^*+\partial^*_J\partial^*\partial \partial_J+\partial \partial_J^* \partial_J \partial^*+\partial_J\partial^* \partial \partial_J^*\,.
\]
For each of these we denote with a calligraphic letter the corresponding space of harmonic forms, thus, for instance, $ \mathcal{H}_\partial^{p,0}(M):=\mathrm{Ker}(\Delta_\partial\vert_{A^{p,0}(M)}) $. It is well known that on a compact K\"ahler manifold the spaces of Dolbeault, Bott-Chern and Aeppli-harmonic forms all coincide. We prove that the analogue result is also true for balanced HKT manifolds:

\begin{theorem}[Propositions \ref{prop:equality-Laplacians}, \ref{prop:equality-laplacian-BC}, \ref{prop:BC=A}]\label{Teor:harmonic}
	On a compact balanced HKT manifold $ M $ the spaces of harmonic forms all coincide:
	\[
	\mathcal{H}^{p,0}_\partial(M)= \mathcal{H}^{p,0}_{\partial_J}(M) = \mathcal{H}^{p,0}_{\mathrm{BC}}(M)= \mathcal{H}^{p,0}_{\mathrm{A}}(M)\,.
	\]
	In particular, there are isomorphisms
	\[
	H^{p,0}_\partial(M)\cong H^{p,0}_{\partial_J}(M) \cong H^{p,0}_{\mathrm{BC}}(M)\cong H^{p,0}_{\mathrm{A}}(M)
	\]
	for every $ p $.
\end{theorem}

We remark that the equality of $ \Delta_{\partial_J} $ and $ \Delta_\partial $ on balanced HKT manifolds is already implicitly proved in \cite[Theorem 10.2]{verbitsky-hodge}. Along the way we shall also study the Hard Lefschetz condition on these spaces (see Theorems \ref{thm: deldellambda-lemma} and \ref{thm:hlc}).

Another interesting notion to be investigated is formality since it provides an obstruction to the $\del\del_J$-lemma and so to the existence of HKT $\mathrm{SL}(n,\mathbb{H})$-structures. More precisely, we prove

\begin{theorem}[Theorem \ref{thm:deldelj-lemma-formal}]\label{Teor:formality}
	Let $(M,I,J,K)$ be a compact hypercomplex manifold satisfying the $\del\del_J$-lemma, then the differential graded algebra $(A^{\bullet,0}(M),\del)$ is formal.
\end{theorem}

As a consequence, we obtain that triple $ \partial $-Massey products vanish, which allows us to derive the following interesting obstruction for a compact complex $ 4n $-dimensional manifold to allow HKT $ \mathrm{SL}(n,\mathbb{H}) $-structures.

\begin{cor}[cf. Corollary \ref{cor:non-existence-hkt}]\label{Cor:Massey}
Let $(M,I)$ be a $4n$-dimensional compact complex manifold with holomorphically trivial canonical bundle and such that there exists a non trivial $\del$-Massey product. Then $(M,I)$ does not admit any complex structures $J,K$ such that $(M,I,J,K)$ is hypercomplex and admits a HKT 
metric.
\end{cor}

Notice that, as a consequence of Theorems \ref{thm:vanishing-massey}, \ref{thm:vanishing-massey-nilmanifold}, on nilmanifolds with invariant HKT structures the $\del$-Massey products vanish. For this reason, 
we use such obstruction to give some examples of complex solvmanifolds without invariant HKT  $\mathrm{SL}(n,\mathbb{H})$-structures.

\medskip
The organization of the paper is the following. In Section \ref{Sec:Lef} we briefly study a class of differential graded algebras which we call ``Lefschetz''. Building on the work of Tomassini and Wang \cite{Tomassini-Wang} we define a generalization of the Hodge star operator, which allows us to take into account formal adjoints and discuss some relations between Laplacians. This algebraic picture is then applied to HKT manifolds in Section \ref{Sec:HKT}, leading us to the proof of Theorem \ref{Teor:harmonic}. Section \ref{Sec:form} is then devoted to investigate the notion of formality. Here we prove Theorem \ref{Teor:formality} and Corollary \ref{Cor:Massey}. Finally, in Section \ref{Sec:solv} we briefly study hypercomplex structures on solvmanifolds and their relation with the balanced condition.

\medskip 
\noindent {\bf Acknowledgements.} Both authors are deeply grateful to Luigi Vezzoni for proposing this research argument and for many helpful discussions. This paper was partly written during the first author's visits to IMPA (Rio de Janeiro) and IMECC/UNICAMP (Campinas). He wishes to thank both institutions for the warm hospitality. In particular, he expresses his most sincere gratitude to Misha Verbitsky for countless discussions, remarkable helpfulness and guidance during the visit experience at IMPA and he is profoundly thankful to Henrique S\'{a} Earp for many instructive conversations. The second author would like to thank Riccardo Piovani, Tommaso Sferruzza and Adriano Tomassini for interesting discussions on formality and Massey products. Both authors would like to thank Adri\'{a}n Andrada and Alejandro Tolcachier for sharing a preliminary version of \cite{Andrada-Tolcachier} and for pointing out an inaccuracy in the first version of Theorem \ref{thm:solv_SLnH}. We also thank the anonymous referee for their comments that improved the presentation of the paper. The second author has financially been supported by the Programme ``FIL-Quota Incentivante'' of University of Parma and co-sponsored by Fondazione Cariparma. This work was partially supported by GNSAGA of INdAM.

\section{Lefschetz spaces}\label{Sec:Lef}

In this section, inspired by the algebraic treatment of Tomassini and Wang \cite{Tomassini-Wang}, we wish to push a little further their work, proving some identities between Laplacians defined in a fairly general context. We start by recalling the main definitions and results from \cite{Tomassini-Wang} (see also \cite{Wang}).

\begin{definition}
Let $ A=\bigoplus_{p=0}^{2n} A^p $ be a direct sum of complex vector spaces. Let $ L $ be a $ \C $-linear endomorphism of $ A $ such that $ L(A^p)\subseteq A^{p+2} $ for $ p=0,\dots,2n-2 $ and $ L(A^{2n-1})=L(A^{2n})=0 $. We say that $ (A,L) $ is a \emph{Lefschetz space} if $ L $ satisfies the Hard Lefschetz Condition (HLC), i.e.
\[
L^{n-p}\colon A^{p}\to A^{2n-p}
\]
is an isomorphism for all $ p=0,\dots,n $.

If a Lefschetz space $ (A,L) $ is equipped with a $ \C $-linear endomorphism $ d $ such that $ d(A^p)\subseteq A^{p+1} $ for $ p=0,\dots,2n-1 $, while $ d(A^{2n})=0 $ we call the triple $ (A,L,d) $ a \emph{differential Lefschetz space}.\\
If moreover $d^2=0$ then the triple $ (A,L,d) $ is called a \emph{Lefschetz complex}.
\end{definition}

On a Lefschetz space we say that $ \alpha \in A^p $ is a \emph{primitive form} if $ p\leq n $ and $ L^{n-p+1}\alpha=0 $. By the HLC immediately follows the decomposition into primitive forms (see \cite{Wang}), more precisely, for every $ \alpha \in A^p $ there exist unique primitive $ \alpha^k\in A^{p-2k} $ such that
\begin{equation}\label{Lef_dec}
\alpha= \sum_{k=0}^{\lfloor p/2 \rfloor} \frac{1}{k!}L^k\alpha^k\,.
\end{equation}

As a generalization of the symplectic star operator Tomassini and Wang \cite{Tomassini-Wang} introduced the \emph{Lefschetz star operator} $ *_L\colon A \to A $, acting on a primitive form $ \beta\in A^p $ as follows:
\begin{equation*}
*_L\frac{1}{k!}L^k\beta:=(-1)^{1+2+\dots+p}\frac{1}{(n-p-k)!}L^{n-p-k}\beta\,.
\end{equation*}
Clearly the definition is then extended by linearity to any $ \alpha \in A^p $ via the Lefschetz decomposition \eqref{Lef_dec}. Notice that $*_L^2=1$.

The starting point of the discussion by Tomassini and Wang is the following general Demailly-Griffiths-K\"ahler identity \cite[Theorem A]{Tomassini-Wang}.
\begin{theorem}\label{DGK-identity}
Let $ (A,L,d) $ be a differential Lefschetz space and $ \Lambda=*_L^{-1} L*_L$ the dual Lefschetz operator. Define $ d^\Lambda\in \mathrm{End}(A) $ by
\[
d^\Lambda\vert_{A^p}:=(-1)^{p+1}*_Ld*_L\,,
\]
and assume that $ [L,[d,L]]=0 $, then
\[
[d^\Lambda,L]=d+[\Lambda,[d,L]]\,,\qquad [d,\Lambda]=d^\Lambda+[[\Lambda,d^\Lambda],L]\,.
\]
\end{theorem}

Notice that if $ (A,L,d) $ is a Lefschetz complex then $d^2=0$ implies $(d^\Lambda)^2=0$. In case $[d,L]=0$, one also obtains that $$
[d,d^\Lambda]=0.
$$
Therefore, on a Lefschetz complex with $[d,L]=0$ one has that the triple $(A,d,d^\Lambda)$ is a double complex.

We summarize here the main consequences which we are interested in (cf. \cite[Theorems 3.3, 3.5]{Tomassini-Wang}).

\begin{theorem}\label{Symplectic}
Let $ (A,L,d) $ be a Lefschetz complex. Suppose $ [d,L]=0 $ and denote with $ \mathcal{H}^p_L $ the space of \emph{Lefschetz harmonic} $ p $-forms, i.e. elements $ \alpha \in A^p $ such that
\[
d \alpha=0=d^\Lambda\alpha.
\]
Then $ (\mathcal{H}^{\bullet}_L,L) $ and  $ (\mathcal{H}^{\bullet}_L,\Lambda) $ satisfy the HLC. Furthermore the following are equivalent:
\begin{itemize}
\item $ (A^{\bullet},L) $ satisfies the $ dd^\Lambda $-lemma, i.e.,
$$
\mathrm{Ker}\,d\cap \mathrm{Ker}\,d^\Lambda\cap
 (\mathrm{Im}\,d+\mathrm{Im}\,d^\Lambda) = \mathrm{Im}\,dd^\Lambda\,
$$
\item There is a Lefschetz harmonic representative in each cohomology class of $ H^{\bullet}_d $;
\item $ (H_d^{\bullet},L) $ satisfies the HLC;
\item $ (H_{d^\Lambda}^{\bullet},\Lambda) $ satisfies the HLC.
\end{itemize}
\end{theorem}

\medskip
Now we wish to introduce in the picture (a generalization of) the Hodge star operator, in order to do so, we need a complex structure on our Lefschetz space.

\begin{definition}
A \emph{Lefschetz (differential) graded algebra} is a (differential) Lefschetz space $ A=\bigoplus_{p=0}^{2n}A^p $ which is also a graded algebra that is generated by $ A^1 $ over $ \C $.
\end{definition}

Let $ A $ be a Lefschetz graded algebra and assume that $ A^1 $ is equipped with an endomorphism $ \mathcal{J} $ such that $\mathcal{J}^2=-I$. We extend the action of $ \mathcal{J} $ on $ A $ by setting on homogeneous elements 
\[
\mathcal{J}(\alpha_1\cdots\alpha_k)=\mathcal{J}\alpha_1 \cdots \mathcal{J}\alpha_k\,, \qquad \text{for every } \alpha_1,\dots,\alpha_k\in A^1\,,
\]
and then extending by $ \C $-linearity.

We make the assumption that $ \mathcal{J} L= L \mathcal{J} $ and consequently introduce a generalization of the Hodge star operator by setting
\begin{equation}\label{Hod}
	*:=\mathcal{J} *_L=*_L \mathcal{J}
\end{equation}
or equivalently
\begin{equation}\label{Hodge}
*\frac{1}{k!}L^k\beta:=(-1)^{1+2+\dots+p}\frac{1}{(n-p-k)!}L^{n-p-k}\mathcal{J} \beta\,,
\end{equation}
for a primitive $ \beta\in A^{p} $ and then extend the definition on all $ A $ by bilinearity via the Lefschetz decomposition \eqref{Lef_dec}. It follows that
\[
*^2\vert_{A^p}=\mathcal{J}^2\vert_{A^p}=(-1)^{p}\,.
\]

\begin{rem}
	Let $ (M,J,\omega) $ be an almost K\"ahler manifold, namely $\omega$ is a symplectic structure on a smooth manifold $M$ and $J$ is a compatible almost complex structure. Clearly when $J$ is integrable (and so $(M,J)$ is a complex manifold) then $ (M,J,\omega) $ is a K\"ahler manifold. Set $ L=\omega \wedge - $ for the usual Lefschetz operator and let $ A=\bigoplus_{p=0}^{2n}A^p $ be the Lefschetz graded algebra of differential forms on $ M $. The almost complex structure $ J\colon TM \to TM $ naturally induces a complex structure $ \mathcal{J} $ on $ A^1 $. Since $ \omega $ is a $ (1,1) $-form we have $ \mathcal{J}L=L\mathcal{J} $ and the description above is coherent with the well-known almost K\"ahler case. Indeed, formula \eqref{Hodge}, where $ * $ is the usual Hodge operator, is sometimes referred to as the Weil relation \cite{Weil}.\\
\end{rem}

Now, suppose $ A $ is equipped with a differential $ d $. Consider the dual Lefschetz operator
\[
\Lambda=*_L^{-1} L*_L=*^{-1}\mathcal{J} L\mathcal{J}^{-1} *= *^{-1} L*
\]
and define as before the ``Lefschetz adjoint'' of $ d $, i.e. $ d^\Lambda\in \mathrm{End}(A) $ given by $ d^\Lambda\vert_{A^p}:=(-1)^{p+1}*_Ld*_L $. Then, by Theorem \ref{DGK-identity}, if $ [d,L]=0 $ we have $ d^\Lambda=[d,\Lambda] $ and $ d=[d^\Lambda,L] $.

We may take the ``Hodge adjoints'' 
\[
d^*=-*d*\,, \qquad d^{\Lambda*}=-*d^{\Lambda}*\,,
\]
and obtain also $ d^*= [\Lambda,d^{\Lambda*}] $.

Now, we consider the following operators and we aim to study the relations between them:
\[
\Delta_d=dd^*+d^*d\,, \qquad \Delta_{d^\Lambda}=d^{\Lambda}d^{\Lambda*}+d^{\Lambda*}d^{\Lambda}\,,
\]
\[
\Delta^{\mathrm{BC}}_{d^\Lambda}=d^*d+d^{\Lambda*}d^\Lambda+d d^\Lambda d^{\Lambda*} d^*+d^{\Lambda*}dd^*d^\Lambda+d^*d^\Lambda d^{\Lambda*}d+d^{\Lambda*}d^*dd^\Lambda\,,
\]
\[
\Delta^{\mathrm{BC}}_{d^{\Lambda*}}=d^*d+d^{\Lambda}d^{\Lambda*}+d d^{\Lambda*} d^{\Lambda} d^*+d^{\Lambda}dd^*d^{\Lambda*}+d^*d^{\Lambda*} d^{\Lambda}d+d^{\Lambda}d^*dd^{\Lambda*}\,.
\]
We will denote with $\mathcal{H}^{\bullet}_{d}$ and $\mathcal{H}^{\bullet}_{d^\Lambda}$ the kernels of $\Delta_d$ and $\Delta_{d^\Lambda}$ respectively. All these operators where originally introduced for symplectic manifolds in \cite{tseng-yau}.

Since it will be useful in the following we recall that for the graded bracket $[A,B]=AB-(-1)^{\text{deg}\,A\,\text{deg}\,B}BA$, the graded Jacobi identity holds
$$
[A,[B,C]]=[[A,B],C]+(-1)^{\text{deg}\,A\,\text{deg}\,B}[B,[A,C]]\,.
$$

\begin{prop}\label{Laplacians}
In the previous assumptions it holds
\[
\Delta_d=\Delta_{d^\Lambda}-[\Lambda,[d,d^{\Lambda*}]]\,.
\]
In particular, if $ [d,d^{\Lambda*}]=0 $ the kernels of $\Delta_d$  and $\Delta_{d^\Lambda}$ coincide, namely for every $p$ we have
$$
\mathcal{H}^{p}_{d}=\mathcal{H}^{p}_{d^\Lambda}\,.
$$
\end{prop}
\begin{proof}
Using $ [\Lambda,d^{\Lambda*}]=d^* $ and $[d,\Lambda]=d^\Lambda$ we obtain
$$
\Delta_{d}=[d,d^*]=[d,[\Lambda,d^{\Lambda*}]]=[[d,\Lambda],d^{\Lambda*}]-[\Lambda,[d,d^{\Lambda*}]]=[d^\Lambda,d^{\Lambda*}]-[\Lambda,[d,d^{\Lambda*}]]=\Delta_{d^\Lambda}-[\Lambda,[d,d^{\Lambda*}]]\,. \qedhere
$$
\end{proof}

\begin{prop}\label{Bott-Chern}
If $ [d,d^{\Lambda*}]=0 =[d,L] $, then
\begin{align*}
\Delta^{\mathrm{BC}}_{d^\Lambda}&= \Delta_{d^\Lambda}\Delta_{d^\Lambda}+d^*d+d^{\Lambda*}d^\Lambda\\
&=\Delta^{\mathrm{BC}}_{d^{\Lambda*}}+d^{\Lambda*}d^\Lambda-d^\Lambda d^{\Lambda*}\,.
\end{align*}
\end{prop}
\begin{proof}
Notice that under our assumptions we also have $[d^*,d^\Lambda]=0$ and $[d^*,d^{\Lambda*}]=0$. We start by considering $ \Delta_{d^\Lambda}\Delta_{d^\Lambda}$. By Proposition \ref{Laplacians}
$$
 \Delta_{d^\Lambda}\Delta_{d^\Lambda}=
\Delta_{d}\Delta_{d^\Lambda}=
dd^*d^\Lambda d^{\Lambda*}+dd^*d^{\Lambda*}d^\Lambda+
d^*dd^\Lambda d^{\Lambda*}+
d^*dd^{\Lambda*}d^\Lambda=(I)+(II)+(III)+(IV)
$$
Now we will treat the four terms separately. Using that $ [d,d^{\Lambda*}]=0 $ and $[d,d^\Lambda]=0$
\begin{align*}
(I) &=dd^*d^\Lambda d^{\Lambda*}=-
dd^\Lambda d^*d^{\Lambda*}=d d^\Lambda d^{\Lambda*} d^*\,,&
(II) &=dd^*d^{\Lambda*}d^\Lambda=-dd^{\Lambda*}d^*d^\Lambda=d^{\Lambda*}dd^*d^\Lambda\,,\\
(III)&=d^*dd^\Lambda d^{\Lambda*}=-d^*d^\Lambda dd^{\Lambda*}=d^*d^\Lambda d^{\Lambda*}d\,, &
(IV)&= d^*dd^{\Lambda*}d^\Lambda=-d^*d^{\Lambda*}dd^\Lambda=d^{\Lambda*}d^*dd^\Lambda \,.
\end{align*}

Now, putting the four terms together we have
$$
 \Delta_{d^\Lambda}\Delta_{d^\Lambda}=
\Delta_{d}\Delta_{d^\Lambda}=
d d^\Lambda d^{\Lambda*} d^*+d^{\Lambda*}dd^*d^\Lambda+d^*d^\Lambda d^{\Lambda*}d+d^{\Lambda*}d^*dd^\Lambda =
\Delta^{\mathrm{BC}}_{d^\Lambda}-d^*d-d^{\Lambda*}d^\Lambda\,.
$$
Furthermore using again that $ [d,d^{\Lambda*}]=0 $ and $[d,d^\Lambda]=0$ we obtain
\begin{align*}
\Delta^{\mathrm{BC}}_{d^\Lambda}=& d^*d+d^{\Lambda*}d^\Lambda+d d^\Lambda d^{\Lambda*} d^*+d^{\Lambda*}dd^*d^\Lambda+d^*d^\Lambda d^{\Lambda*}d+d^{\Lambda*}d^*dd^\Lambda\\
=& d^*d+d^{\Lambda*}d^\Lambda+d^\Lambda d  d^*d^{\Lambda*}+dd^{\Lambda*}d^\Lambda d^*+d^\Lambda d^*d d^{\Lambda*}+d^*d^{\Lambda*}d^\Lambda d\\
=& d^{\Lambda*}d^\Lambda-d^\Lambda d^{\Lambda*}+\Delta^{\mathrm{BC}}_{d^{\Lambda*}}
\end{align*}
as desired.
\end{proof}

\begin{rem}
Let $(M,\omega,J,g)$ be an almost-K\"ahler manifold, then
$[d,d^{\Lambda*}]=0$ if and only if $[d,d^c]=0$ if and only if $J$ is integrable. In such a case $(M,\omega,J,g)$ is a K\"ahler manifold and we recover the usual equalities for the Laplacians in
Propositions \ref{Laplacians} and \ref{Bott-Chern}.
\end{rem}

\section{Hodge theory on HKT manifolds}\label{Sec:HKT}

\subsection{Lefschetz spaces on HKT manifolds}
We specialize the results of the previous section to the cohomology of HKT manifolds.

Let $(M,I,J,K)$ be a $4n$-dimensional compact hypercomplex manifold and $\Omega\in A^{2,0}(M)$ a non-degenerate $(2,0)$-form on $(M,I)$. Then, as usual we set
$$
L\colon A^{r,0}(M)\to A^{r+2,0}(M)\,,\qquad
L:=\Omega\wedge-
$$
for the Lefschetz operator.\\
Then $(A^{\bullet,0}(M), L)$ is a Lefschetz space.\\
Moreover, if we consider as differential operator $\del$ (always taken with respect to $I$), since $I$ is integrable, $\del^2=0$ and so $(A^{\bullet,0}(M), L,\del)$ defines a Lefschetz complex. If $\del\Omega=0$ then we have
$$
[\del,L]=0.
$$

Denote with $ \mathcal{H}^{p,0}_L(M) $ the space of \emph{Lefschetz harmonic} $ (p,0) $-forms, i.e. forms $ \alpha \in A^{p,0}(M) $ such that $ \partial \alpha=0=\partial^\Lambda\alpha $, where $ \partial^\Lambda=[\partial,\Lambda] $.

 We can therefore apply the results of the previous section to infer

\begin{theorem}\label{thm: deldellambda-lemma}
Let $(M,I,J,K,\Omega)$ be a compact $4n$-dimensional hypercomplex manifold and $\Omega\in A^{2,0}(M)$ a non-degenerate $(2,0)$-form on $(M,I)$ such that $\del\Omega=0$.
Then $ (\mathcal{H}^{\bullet,0}_L(M),L) $ and  $ (\mathcal{H}^{\bullet,0}_L(M),\Lambda) $ satisfy the HLC. Furthermore the following are equivalent:
\begin{itemize}
\item $ (A^{\bullet,0}(M),L) $ satisfies the $ \partial\partial^\Lambda $-lemma, i.e.,
$$
\mathrm{Ker}\,\del\cap \mathrm{Ker}\,\del^\Lambda\cap
 (\mathrm{Im}\,\del+\mathrm{Im}\,\del^\Lambda) = \mathrm{Im}\,\del\del^\Lambda\,;
$$

\item There is a Lefschetz harmonic representative in each Dolbeault cohomology class of $ H^{\bullet,0}_\partial(M) $;
\item $ (H_\partial^{\bullet,0}(M),L) $ satisfies the HLC;
\item $ (H_{\partial^\Lambda}^{\bullet,0}(M),\Lambda) $ satisfies the HLC.
\end{itemize}
\end{theorem}

Moreover, by the general results in the previous section we obtain
\begin{prop}
Let $(M,I,J,K,\Omega)$ be a compact $4n$-dimensional hypercomplex manifold and $\Omega\in A^{2,0}(M)$ a non-degenerate $(2,0)$-form on $(M,I)$ such that $\del\Omega=0$.
Then,
\[
\Delta_{\del}=\Delta_{\del^\Lambda}-[\Lambda,[\del,\del^{\Lambda*}]]\,.
\]
In particular, if $ [\del,\del^{\Lambda*}]=0 $ 
\[
\Delta_{\del}=\Delta_{\del^\Lambda},
\]
and for every $p$ we have
$$
\mathcal{H}^{p,0}_{\del}(M)=\mathcal{H}^{p,0}_{\del^\Lambda}(M)\,.
$$
Moreover, if $ [\del,\del^{\Lambda*}]=0$
$$
\Delta^{\mathrm{BC}}_{\del^\Lambda}=\Delta^{\mathrm{BC}}_{\del^{\Lambda*}}+\del^{\Lambda*}\del^\Lambda-\del^\Lambda \del^{\Lambda*}= \Delta_{\del^\Lambda}\Delta_{\del^\Lambda}+\del^*\del+\del^{\Lambda*}\del^\Lambda\,.
$$
\end{prop}

\subsection{Complex Hodge theory on HKT manifolds}

If we further assume that $ \Omega $ is q-positive, in the sense that $ J\Omega=\bar \Omega $ and $ \Omega(Z,J\bar Z)>0 $ for every $ Z\in T^{1,0}_IM $, $ Z\neq 0 $, then it must be the HKT form corresponding to a HKT metric $ g $ on $ (M,I,J,K) $. If $(M,I,J,K,g, \Omega)$ is HKT by \cite{verbitsky-hodge} we have
$$
[\partial^*,L]=-\partial_J+\theta_J\wedge -
$$
where $\theta_J$ is a $1$-form defined as follows. Since $\Omega$ is non-degenerate, there exists a $(0,1)$-form $\bar\theta$ such that
$$
\delbar\Omega^n=\bar\theta\wedge\Omega^n\,
$$
then, by definition, $\theta_J:=J\bar\theta\in A^{1,0}(M)$.
Notice that $(M,I,J,K,g, \Omega)$ is balanced if and only if $\theta_J=0$
and so for balanced HKT manifolds we have
$$
[\partial^*,L]=-\partial_J.
$$
Moreover recall that a balanced HKT manifold has holonomy in $\mathrm{SL}(n,\mathbb{H})$. Denoting with $\Lambda$ the adjoint of $L$, one can easily get by duality or by applying $J$ the following (see \cite{verbitsky-hodge})
\begin{prop}\label{prop:HKT_identities}
Let $(M,I,J,K,g, \Omega)$ be a compact balanced HKT manifold. Then,
the following identities hold
\begin{itemize}
\item $[\partial^*,L]=-\partial_J\,,$
\item $[\del,\Lambda]=\del_J^*\,,$
\item $[L,\del_J^*]=-\partial\,,$
\item $[\Lambda,\del_J]=\del^*\,.$
\end{itemize}
\end{prop}

Now, we shall show that the framework of the previous section can be used to study quaternionic cohomologies. First of all, we set $ \mathcal{J} \alpha=J\bar\alpha $ for every $ \alpha \in A^{1,0}(M) $, thus $ \mathcal{J} $ is a complex structure on $ A^{1,0}(M) $ and naturally extends to $ A^{p,0}(M) $ by imposing compatibility with the wedge product. Since $ \Omega $ is q-real we have $ \mathcal{J}L= L\mathcal{J} $ and we can use \eqref{Hod} to define a Hodge-type operator.

We warn the reader that in this framework the operator defined by \eqref{Hod} slightly differs from the usual Hodge operator. To distinguish them, let us denote here $*:A^{p,0}(M)\to A^{2n-p,0}(M) $ the operator defined in \eqref{Hod} and $ \hat *:A^{p,q}(M)\to A^{2n-p,2n-q}(M) $ the usual Hodge star operator, then one can easily show that for every $ \alpha,\beta \in A^{p,0}(M) $
\[
\alpha \wedge *\beta=g(\alpha,\beta) \frac{\Omega^n}{n!}\,,
\]
where $ g $ here is the Hermitian product induced by the Riemannian metric on $ A^{p,0}(M) $, while, by definition,
\[
\alpha \wedge \hat * \beta=g(\alpha,\beta) \frac{\Omega^n\wedge \bar \Omega^n}{(n!)^2}\,.
\]
However, we can identify the formal adjoints of $ \partial $ and $ \partial_J $ with respect to $ * $ and $ \hat * $ in the following way.

Suppose $ M $ is a $ \mathrm{SL}(n,\mathbb{H}) $-manifold and fix a q-positive q-real holomorphic $ (2n,0) $-form  $ \Theta $. Define the following $ L^2 $-products:
	\[
	(\alpha,\beta)_1:=\int_M g(\alpha,\beta)\frac{\Omega^n\wedge \bar \Omega^n}{(n!)^2}=\int_M\alpha \wedge \hat * \beta\,, \qquad (\alpha,\beta)_2:=\int_M g(\alpha,\beta)\frac{\Omega^n}{n!}\wedge \bar\Theta=\int_M\alpha \wedge *\beta \wedge \bar \Theta\,.
	\]
Then the adjoint of $ \partial $ and $ \partial_J $ with respect to $ (\cdot,\cdot)_1 $ are $ \partial^{\hat *}=-\hat * \partial \hat *$ and $ \partial_J^{\hat *}=-\hat *  \partial_J \hat * $, while those with respect to $ (\cdot,\cdot)_2 $ are $\partial^*=-*\partial * $ and $ \partial_J^*=-*\partial_J * $ (cf. \cite{lejmi-weber}). 
Since $ \Theta $ is q-positive, there exists a real-valued function $ f>0 $ such that $ \Theta=f\frac{\Omega^n}{n!} $, moreover, the holomorphicity of $ \Theta $ translates into the condition $ \partial f+ f\theta=0 $. Now, observe that $ (\cdot,\cdot)_2=(f\cdot,\cdot)_1 $ thus
	\[
	(\alpha,\partial^*\beta)_2=(\partial \alpha,\beta)_2=(f\partial \alpha,\beta)_1=(\partial(f\alpha)-\partial f \wedge \alpha,\beta)_1=(\alpha,\partial^{\hat *}\beta)_2+(\theta \wedge \alpha, \beta)_2
	\]
and similarly, working with $ \partial_J^* $ and $ \partial_J^{\hat *} $ one obtains
	\[
	(\alpha,\partial_J^*\beta)_2=(\alpha,\partial^{\hat *}_J\beta)_2-(\theta_J \wedge \alpha, \beta)_2\,.
	\]
In particular if $ M $ is balanced then $ \theta=\theta_J=0 $ and $ f $ is constant, so that the two $ L^2 $-products coincide up to a constant and $ \partial^*=\partial^{\hat{*}} $ and $ \partial^*_J=\partial^{\hat{*}}_J $. In particular the usual Laplacians obtained by means of the Riemannian Hodge star operator coincide with those Laplacians considered in Section 2 and the related results can be applied.

In particular, if $M$ is compact and if $\alpha\in A^{p,0}(M)$ one immediately obtains
\begin{equation*}
\begin{cases}
\alpha\in\mathcal{H}^{p,0}_{\partial}(M)\ &\iff\  \partial\alpha=0,\quad\partial^*\alpha=0;\\
\alpha\in\mathcal{H}^{p,0}_{\partial_J}(M)\ &\iff\ \partial_J\alpha=0,\quad\partial_J^*\alpha=0;\\
\alpha\in\mathcal{H}^{p,0}_{BC}(M)  \ &\iff\ \partial\alpha=0,\quad\partial_J\alpha=0,\quad \partial_J^*\del^*\alpha=0;\\
\alpha\in\mathcal{H}^{p,0}_{A}(M)\ &\iff \ \partial^*\alpha=0,\quad\partial_J^*\alpha=0,\quad \partial\del_J\alpha=0.
\end{cases}
\end{equation*}

Proposition \ref{prop:HKT_identities} shows that $ \partial^\Lambda=\partial_J^* $ and we readily obtain from Proposition \ref{Laplacians}.

\begin{prop}\label{prop:equality-Laplacians}
Let $(M,I,J,K,\Omega)$ be a compact balanced HKT manifold, then
$$
\Delta_{\partial_J}=\Delta_{\partial}\,.
$$
In particular, the spaces of harmonic forms coincide, namely for every $p$ we have
$$
\mathcal{H}^{p,0}_{\partial_J}(M)=\mathcal{H}^{p,0}_{\partial}(M)\,.
$$
\end{prop}

\begin{rem}
If we do not assume the compact HKT manifold $(M,I,J,K,\Omega)$ to be balanced we would have, in general
$$
[\partial^*,L]=-\partial_J+\theta_J\wedge-\,.
$$
Setting $\tau(\alpha):=\theta_J\wedge\alpha$ and $\psi(\alpha):=\theta\wedge\alpha$, then, we would get
$$
\partial^\Lambda=[\partial,\Lambda]=\partial_J^*-\tau^*\,.
$$
In particular $ [\partial,\partial^{\Lambda*}]=-\partial \theta_J \wedge - $ and in such a case the Laplacians $ \Delta_\partial $, $ \Delta_{\partial^\Lambda} $ and $\Delta_{\partial_J}$ do not coincide, and in fact by a direct computation one gets
$$
\Delta_{\partial_J}=\Delta_{\partial}+[\psi^*,\del]+[\del_J,\tau^*].
$$
Notice that $\psi^*=\iota_{\theta^\sharp}$ and $\tau^*=\iota_{\theta_J^\sharp}$.

We also observe that the condition $ [\partial,\partial^{\Lambda*}]=0 $, i.e. $ \partial \theta_J=0 $ is only satisfied when the manifold is balanced, indeed $ \partial \theta_J=0 $ is equivalent to $ \partial_J \theta=0 $, from which we obtain
$$
\begin{aligned}
\partial_J \partial \bar \Omega^n=&\,\partial_J(\theta \wedge \bar \Omega^n)=\partial_J \theta \wedge \bar \Omega^n-\theta \wedge \partial_J \bar \Omega^n=-\theta \wedge J^{-1} \bar \partial \Omega^n\\
=&\,-\theta \wedge J^{-1} (\bar \theta \wedge \Omega^n)=\theta \wedge J\bar \theta \wedge \bar \Omega^n\,.
\end{aligned}
$$
Therefore by integrating we infer
\[
0=\int_M \partial_J \bar \Omega^n \wedge \partial \Omega^{n-1}=\int_M \partial_J \partial \bar \Omega^n \wedge \Omega^{n-1}=\int_M \theta \wedge J\bar \theta \wedge \bar \Omega^n \wedge \Omega^{n-1}=\frac{1}{n}\int_M \|\theta\|^2\, \Omega^n \wedge \bar \Omega^n
\]
and the claim follows. 
\end{rem}

\begin{rem}
We can also reinterpret Proposition \ref{prop:equality-Laplacians} as follows.
We first notice that if $(M,I,J,K,\Omega)$ is a compact balanced HKT manifold. Then,
$[\del,\del^{\Lambda*}]=-[\del,\del_J]=0$.\\
Hence, by Propositions \ref{Laplacians} and \ref{Bott-Chern} we obtain
that
\[
\Delta_{\del}=\Delta_{\del^\Lambda}
\]
and for every $p$ we have
$$
\mathcal{H}^{p}_{\del}(M)=\mathcal{H}^{p}_{\del^\Lambda}(M)\,.
$$
Moreover, 
$$
\Delta^{\mathrm{BC}}_{\del^\Lambda}= \Delta_{\del^\Lambda}\Delta_{\del^\Lambda}+\del^*\del+\del^{\Lambda*}\del^\Lambda\,=
\Delta^{\mathrm{BC}}_{\del^{\Lambda*}}+\del^{\Lambda*}\del^\Lambda-\del^\Lambda \del^{\Lambda*}.
$$
\end{rem}

As a consequence of Proposition \ref{prop:equality-Laplacians} we obtain isomorphisms for the associated cohomology groups (cf. \cite[Proposition 2.3]{lejmi-weber} where it is noticed that an isomorphism, induced by $J$ and conjugation with respect to $I$, holds in general for hypercomplex manifolds).

\begin{cor}
Let $(M,I,J,K,\Omega)$ be a compact HKT balanced manifold, then
$$
H^{p,0}_{\partial_J}(M)\simeq H^{p,0}_{\partial}(M).
$$
In particular, we have the equalities
$$
h^{p,0}_{\partial_J}(M)=h^{p,0}_{\partial}(M).
$$
\end{cor}

Invoking Proposition \ref{Bott-Chern} we obtain that, similarly to the K\"ahler case, the Laplacians $\Delta_{\mathrm{BC}}$ and $\Delta_{\partial_J}=\Delta_{\partial}$ are related.
\begin{prop}\label{prop:equality-laplacian-BC}
Let $(M,I,J,K,\Omega)$ be a compact HKT balanced manifold, then
\begin{align*}
\Delta_{\mathrm{BC}}&=\Delta_{\partial_J}\Delta_{\partial_J}+\partial^*\partial+\partial_J^*\partial_J\\
&= \Delta_{\partial}\Delta_{\partial}+\partial^*\partial+\partial_J^*\partial_J\,.
\end{align*}
In particular, the spaces of harmonic forms coincide, namely for every $p$ we have
$$
\mathcal{H}^{p,0}_{\mathrm{BC}}(M)=\mathcal{H}^{p,0}_{\partial_J}(M).
$$
\end{prop}

Consequently we obtain isomorphisms for the associated cohomology groups.
\begin{cor}
Let $(M,I,J,K,\Omega)$ be a compact balanced HKT manifold, then for every $p$,
$$
H^{p,0}_{\mathrm{BC}}(M)\simeq H^{p,0}_{\partial_J}(M)\simeq H^{p,0}_{\partial}(M).
$$
In particular, we have the equalities
$$
h^{p,0}_{\mathrm{BC}}(M)=h^{p,0}_{\partial_J}(M)=h^{p,0}_{\partial}(M).
$$
\end{cor}

Notice that these results are the analogue of the ones proved in \cite{schweizer} for compact K\"ahler manifolds.\\

As a consequence of the previous results we prove that, under the same hypothesis, the Hard Lefschetz condition holds for the cohomologies $H^{\bullet,0}_{\partial}(M)$, $H^{\bullet,0}_{\partial_J}(M)$, $H^{\bullet,0}_{\mathrm{BC}}(M)$, thus generalizing \cite[Proposition 4.7]{barberis-dotti-verbitsky}.
\begin{theorem}\label{thm:hlc}
Let $(M,I,J,K,\Omega)$ be a compact $4n$-dimensional balanced HKT manifold, then for every $i$,
$$
L^{n-i}:\mathcal{H}^{i,0}_{\partial}(M)\to \mathcal{H}^{2n-i,0}_{\partial}(M),
$$
$$
L^{n-i}:\mathcal{H}^{i,0}_{\partial_J}(M)\to \mathcal{H}^{2n-i,0}_{\partial_J}(M)\,,
$$
$$
L^{n-i}:\mathcal{H}^{i,0}_{\mathrm{BC}}(M)\to \mathcal{H}^{2n-i,0}_{\mathrm{BC}}(M)\,\quad
$$
are isomorphisms.\\
In particular,
$$
h^{i,0}_{\mathrm{BC}}=h^{i,0}_{\partial}=h^{i,0}_{\partial_J}=
h^{2n-i,0}_{BC}=h^{2n-i,0}_{\partial}=h^{2n-i,0}_{\partial_J}.
$$
\end{theorem}
\begin{proof}
In view of Propositions \ref{prop:equality-Laplacians}, \ref{prop:equality-laplacian-BC} it is sufficient to prove that
$$
L^{n-i}:\mathcal{H}^{i,0}_{\partial}(M)\to \mathcal{H}^{2n-i,0}_{\partial}(M)
$$
are isomorphisms.
Notice that by hypothesis $\partial\Omega=\partial_J\Omega=0$, hence
$$
[\partial,L]=0\,,\qquad
[\partial_J,L]=0\,.
$$
Let $\alpha\in \mathcal{H}^{i,0}_{\partial}(M)=\mathcal{H}^{i,0}_{\partial_J}(M)$. Then,
$$
\partial\alpha=0\,,\quad
\partial_J\alpha=0\,,\quad
\partial^*\alpha=0\,,\quad
\partial^*_J\alpha=0.
$$
As a consequence
$$
\partial (L^{n-i}\alpha)=L^{n-i}\partial\alpha=0,
$$
and, using $[\partial^*,L]=-\partial_J$, 
$$
\partial^*(L^{n-i}\alpha)=L^{n-i}\partial^*\alpha-
(n-i)L^{n-i-1}\partial_J\alpha=0.
$$
Hence, $L^{n-i}\alpha\in\mathcal{H}^{2n-i,0}_{\partial}(M)$. 
The result follows from $\Omega$ being non-degenerate. 
\end{proof}

Notice that combining this result with Theorem \ref{thm: deldellambda-lemma} we have
\begin{cor}
Let $(M,I,J,K,\Omega)$ be a compact $4n$-dimensional balanced HKT manifold, then it satisfies the $\del\del^\Lambda$-lemma and there exists a Lefschetz harmonic representative in each Dolbeault cohomology class of $H^{\bullet,0}_{\del}(M)$.
\end{cor}

\begin{rem}
Notice that, in general, on  a compact $4n$-dimensional HKT manifold $(M,I,J,K,\Omega)$ we cannot expect the HLC for $H^{i,0}_{\partial}(M)$. Indeed, there are examples of HKT manifolds with $K_{(M,I)}$ holomorphically non-trivial and $H^{2n,0}_{\del}\simeq H^{0,2n}_{\delbar}\simeq H^{2n,0}_{\delbar}=\left\lbrace 0\right\rbrace$ (e.g. quaternionic Hopf surfaces).\\
But HLC would imply that
$$
L^{n}:\mathcal{H}^{0,0}_{\partial}(M)\to \mathcal{H}^{2n,0}_{\partial}(M),
$$
is an isomorphism, which is absurd.
\end{rem}

\begin{prop}\label{prop:BC=A}
Let $(M,I,J,K,\Omega)$ be a compact $4n$-dimensional balanced HKT manifold, then for every $p$ we have
$$
\mathcal{H}^{p,0}_{\mathrm{BC}}(M)=\mathcal{H}^{p,0}_{\mathrm{A}}(M).
$$
\end{prop}
\begin{proof}
We first show the inclusion $\mathcal{H}^{p,0}_{\mathrm{BC}}(M)\subseteq
\mathcal{H}^{p,0}_{\mathrm{A}}(M)$. Let $\alpha\in\mathcal{H}^{p,0}_{\mathrm{B}C}(M)$. By Propositions \ref{prop:equality-Laplacians}, \ref{prop:equality-laplacian-BC}
$\alpha\in\mathcal{H}^{p,0}_{\partial}(M)=\mathcal{H}^{p,0}_{\partial_J}(M)$, namely
$$
\partial\alpha=0\,,\quad
\partial_J\alpha=0\,,\quad
\partial^*\alpha=0\,,\quad
\partial^*_J\alpha=0.
$$
Hence, $\alpha\in \mathcal{H}^{p,0}_{\mathrm{A}}(M)$.
The opposite inclusion $\mathcal{H}^{p,0}_{\mathrm{A}}(M)\subseteq
\mathcal{H}^{p,0}_{\mathrm{BC}}(M)$ follows from Theorem \ref{thm:hlc} and \cite[Remark 21]{grantcharov-lejmi-verbitsky}, indeed for every $p$,
\[
h^{p,0}_{\mathrm{BC}}(M)=h^{2n-p,0}_{\mathrm{BC}}(M)=h^{p,0}_\mathrm{A}(M).\qedhere
\]
\end{proof}

As a corollary we have

\begin{cor}
Let $(M,I,J,K,\Omega)$ be a compact balanced HKT manifold, then for every $p$,
$$
H^{p,0}_{\mathrm{BC}}(M)\simeq H^{p,0}_{\mathrm{A}}(M).
$$
\end{cor}

\section{Formality of HKT manifolds}\label{Sec:form}

It is well known that formality in the sense of Sullivan is an obstruction to K\"ahlerianity, more precisely compact complex manifolds satisfying the $\partial\overline\partial$-lemma are formal (see \cite{deligne-griffiths-morgan-sullivan}).
However, notice that the HKT condition does not imply formality, indeed there are examples of non tori nilmanifolds that are HKT but it is well know that  non tori nilmanifolds are not formal in the sense of Sullivan \cite{hasegawa}.\\

In this section we study formality for compact hypercomplex manifolds. We first recall some definitions.\\
Let $(\mathcal{A},d_{\mathcal{A}})$ and $(\mathcal{B},d_{\mathcal{B}})$
be two differential graded algebras (DGA for short) over a field $\mathbb{K}$. A \emph{DGA-homomorphism} between $\mathcal{A}$ and $\mathcal{B}$
is a $\mathbb{K}$-linear map $f:\mathcal{A}\longrightarrow\mathcal{B}$ such
that
\begin{itemize}
\item[i)] $f(\mathcal{A}^i)\subset\mathcal{B}^i$;
\item[ii)] $f(\alpha\cdot\beta)=f(\alpha)\cdot f(\beta)$;
\item[iii)] $d_{\mathcal{B}}\circ f= f\circ d_{\mathcal{A}}$.
\end{itemize}
Any DGA-homomorphism $f:(\mathcal{A},d_{\mathcal{A}})\longrightarrow(\mathcal{B},
d_{\mathcal{B}})$ induces a DGA-homomorphism in cohomology
$$
H(f):(H^{\bullet}(\mathcal{A},d_{\mathcal{A}}),0)\longrightarrow (H^{\bullet}(
\mathcal{B},d_{\mathcal{B}}),0)\,.
$$
A DGA-homomorphism $f:(\mathcal{A},d_{\mathcal{A}})\longrightarrow(\mathcal{B},
d_{\mathcal{B}})$ is called \emph{quasi-isomorphism} if $H(f)$ is an
isomorphism.\\
Two DGA $(\mathcal{A},d_{\mathcal{A}})$ and $(\mathcal{B},d_{\mathcal{B}})$
are said to be \emph{equivalent} if there exists a sequence of
quasi-isomorphisms of the following form:
\[
\begin{array}{lclclclclclcl}
 & & (\mathcal{C}_1,d_{\mathcal{C}_1})&&&&\cdots
 &&&&(\mathcal{C}_n,d_{\mathcal{C}_n})&&\\
 &\swarrow &&\searrow &&\swarrow &&\searrow &&\swarrow &&\searrow &\\
(\mathcal{A},d_{\mathcal{A}})&&&&(\mathcal{C}_2,d_{\mathcal{C}_2})
&&&&\cdots &&&& (\mathcal{B},d_{\mathcal{B}}).
\end{array}
\]
A DGA $(\mathcal{A},d_{\mathcal{A}})$ is called \emph{formal} if
$(\mathcal{A},d_{\mathcal{A}})$ is equivalent to a DGA
$(\mathcal{B},d_{\mathcal{B}}=0)$.

We show now that for a compact hypercomplex manifold $M$ instead of $(A^\bullet(M),d)$, the appropriate DGA to consider in this context is $(A^{\bullet,0}(M),\del)$ by proving the following

\begin{theorem}\label{thm:deldelj-lemma-formal}
Let $(M,I,J,K)$ be a compact hypercomplex manifold satisfying the $\del\del_J$-lemma, then the DGA $(A^{\bullet,0}(M),\del)$ is formal.
\end{theorem}

In order to prove this Theorem, we will need three lemmas.

\begin{lemma}\label{lemma1}
Let $(M,I,J,K)$ be a compact hypercomplex manifold satisfying the $\del\del_J$-lemma, then the natural inclusion
$$
i:\left( A^{\bullet,0}(M)\cap\mathrm{Ker}\,\del_J,\, \del \right)
\to\left(A^{\bullet,0}(M),\, \del\right) 
$$
is a DGA quasi-isomorphism.
\end{lemma}

\begin{proof}
Notice that $\left( A^{\bullet,0}(M)\cap\mathrm{Ker}\,\del_J,\, \del \right)$ is a DGA and the inclusion
$$
i:\left( A^{\bullet,0}(M)\cap\mathrm{Ker}\,\del_J,\, \del \right)
\to\left(A^{\bullet,0}(M),\, \del\right) 
$$
is a morphism of DGAs. We are left to prove that the map induced in cohomology
$$
H_{\del}(i):H_{\del}\left( A^{\bullet,0}(M)\cap\mathrm{Ker}\,\del_J,\, \del \right)
\to H^{\bullet,0}_{\del}(M)
$$
is an isomorphism.\\
We first prove that $H_{\del}(i)$ is injective. Fix $k$, and let $[\alpha]\in H_{\del}\left( A^{k,0}(M)\cap\mathrm{Ker}\,\del_J,\, \del \right)$ such that
$H_{\del}(i)([\alpha])=[\alpha]_{\del}=0$, hence
$$
\alpha\in \mathrm{Ker}\,\del_J\cap\mathrm{Im}\,\del=\mathrm{Im}\,\del\del_J
$$
i.e., $\alpha=\del\left(\del_J\beta\right)$ for some form $\beta\in A^{k-2,0}(M)$ and clearly $\del_J\beta\in A^{k-1,0}(M)\cap\mathrm{Ker}\,\del_J$, hence
$$
[\alpha]=0\in H_{\del}^{k,0}\left( A^{\bullet,0}(M)\cap\mathrm{Ker}\,\del_J,\, \del \right)
$$
and so $H_{\del}(i)$ is injective.\\
We now prove that $H_{\del}(i)$ is is surjective. Let $a\in H^{k,0}_{\del}(M)$, $a=[\alpha]$ with $\del\alpha=0$. Consider,
$$
\del_J\alpha\in \mathrm{Im}\,\del_J\cap\mathrm{Ker}\,\del=\mathrm{Im}\,\del_J\del
$$
hence $\del_J\alpha=\del_J\del\beta$ for some $\beta$. Therefore,
$\del_J(\alpha-\del\beta)=0$ and $\del(\alpha-\del\beta)=\del\alpha=0$. This means that $\alpha-\del\beta$ defines a class in 
$H_{\del}^{k,0}\left( A^{\bullet,0}M\cap\mathrm{Ker}\,\del_J,\, \del \right)$ and
$$
H_{\del}(i)([\alpha-\del\beta])=[\alpha-\del\beta]_{\del}=
[\alpha]=a\,,
$$
concluding the proof.
\end{proof}

\begin{lemma}\label{lemma2}
Let $(M,I,J,K)$ be a compact hypercomplex manifold satisfying the $\del\del_J$-lemma, then the natural projection
$$
p:\left( A^{\bullet,0}(M)\cap\mathrm{Ker}\,\del_J,\, \del \right)
\to\left(H_{\del_J}^{\bullet,0}(M),\, \del\right) 
$$
is a DGA quasi-isomorphism.
\end{lemma}

\begin{proof}
Notice that the projection
$$
p:\left( A^{\bullet,0}(M)\cap\mathrm{Ker}\,\del_J,\, \del \right)
\to\left(H_{\del_J}^{\bullet,0}(M),\, \del\right) 
$$
is a morphism of DGAs. We are left to prove that the map induced in cohomology
$$
H_{\del}(p):H_{\del}^{\bullet,0}\left( A^{\bullet,0}(M)\cap\mathrm{Ker}\,\del_J,\, \del \right)
\to H^{\bullet,0}_{\del}\left(H_{\del_J}^{\bullet,0}(M),\, \del\right) 
$$
is an isomorphism.\\
We first prove that $H_{\del}(p)$ is injective. Fix $k$, and let $[\alpha]\in H_{\del}^{k,0}\left( A^{\bullet,0}(M)\cap\mathrm{Ker}\,\del_J,\, \del \right)$ such that
$H_{\del}(p)([\alpha])=0$. Hence,
$$
\alpha\in \mathrm{Im}\,\del\cap\mathrm{Ker}\,\del_J=\mathrm{Im}\,\del\del_J
$$
i.e., $\alpha=\del\left(\del_J\beta\right)$ for some form $\beta\in A^{k-2,0}(M)$ and clearly $\del_J\beta\in A^{k-1,0}(M)\cap\mathrm{Ker}\,\del_J$, hence
$$
[\alpha]=0\in H^{\bullet,0}_{\del}\left(H_{\del_J}^{\bullet,0}(M),\, \del\right) 
$$
and so $H_{\del}(p)$ is injective.\\
The surjectivity of $H_{\del}(p)$ is immediate.
\end{proof}

\begin{lemma}\label{lemma3}
Let $(M,I,J,K)$ be a compact hypercomplex manifold satisfying the $\del\del_J$-lemma, then $\del$ is the trivial operator on $H_{\del_J}^{\bullet,0}(M)$.
\end{lemma}

\begin{proof}
Fix $k$ and let $a=[\alpha]_{\del_J}\in H^{k,0}_{\del_J}(M)$, namely $\del_J\alpha=0$. Now
$$
\del a=[\del\alpha]_{\del_J}
$$
and
$$
\del\alpha\in \mathrm{Im}\,\del\cap\mathrm{Ker}\,\del_J=\mathrm{Im}\,\del_J\del
$$
so $\del\alpha=\del_J\del\beta$ for some $\beta$, giving $\del a=[\del_J\del\beta]_{\del_J}=0\in H^{k+1,0}_{\del_J}(M)$, concluding the proof.
\end{proof}

Now we are able to prove Theorem \ref{thm:deldelj-lemma-formal}.

\begin{proof}
Under the assumptions and as a consequence of Lemmas \ref{lemma1}, \ref{lemma2}, \ref{lemma3} we have the following diagram of quasi-isomorphisms of DGAs,
$$
\xymatrix{
      & \left( A^{\bullet,0}(M)\cap\mathrm{Ker}\,\del_J,\, \del \right) \ar[ld]_{i}^{\text{qis}} \ar[rd]^{p}_{\text{qis}} & \\
     \left(A^{\bullet,0}(M),\, \del\right) & &
      \left(H^{\bullet,0}_{\del_J}(M),\, 0\right)
     }. 
$$
hence, by definition, $\left(A^{\bullet,0}(M),\, \del\right)$ is a formal DGA.
\end{proof}

As a consequence of \cite[Theorem 6]{grantcharov-lejmi-verbitsky} and Theorem \ref{thm:deldelj-lemma-formal} we obtain
\begin{cor}
Let $(M,I,J,K,\Omega)$ be a compact HKT $\mathrm{SL}(n,\mathbb{H})$-manifold, then the DGA $(A^{\bullet,0}(M),\del)$ is formal.
\end{cor}

We recall now the definition of triple Massey products of a DGA in our setting.

\begin{definition}
Let
$\mathfrak{a}=\left[\alpha\right]\in H^{p,0}_\del(M)$,
$\mathfrak{b}=\left[\beta\right]\in H^{q,0}_\del(M)$ and
$\mathfrak{c}=\left[\gamma\right]\in H^{r,0}_\del(M)$ such that
$\mathfrak{a} \cup \mathfrak{b}=0\in H^{p+q,0}_\del(M)$ and
$\mathfrak{b} \cup \mathfrak{c}=0\in H^{q+r,0}_\del(M)$;
more precisely suppose that $\alpha\wedge\beta=\del \lambda$ and
$\beta\wedge\gamma=\del \mu$ for some $\lambda\in A^{p+q-1,0}$, $\mu\in 
A^{q+r-1,0}$. The \emph{triple $\del$-Massey product} of
$\mathfrak{a},\mathfrak{b},\mathfrak{c}$ is defined as
\[
\left\langle\mathfrak{a},\mathfrak{b},\mathfrak{c}\right\rangle:=
\left[\lambda\wedge\gamma-(-1)^p\alpha\wedge \mu\right]\in 
\frac{H^{p+q+r-1,0}_{\del}(M)}
{H^{p+q-1,0}_{\del}(M)\cup H^{r,0}_{\del}(M)
+H^{p,0}_{\del}(M)\cup H^{q+r-1,0}_{\del}(M)}.
\] 
\end{definition}

Then, since for a formal DGA the associated Massey products vanish we have the following
\begin{cor}\label{cor:triple_Mass_prod}
Let $(M,I,J,K)$ be a compact hypercomplex manifold satisfying the $\del\del_J$-lemma, then
the triple $\del$-Massey products vanish.
\end{cor}
Hence, we have
\begin{theorem}\label{thm:triple_Mass_prod}
Let $(M,I,J,K,\Omega)$ be a compact HKT $\mathrm{SL}(n,\mathbb{H})$-manifold, then the triple $\del$-Massey products vanish.
\end{theorem}
In particular, triple $\del$-Massey products are an obstruction to the existence of a HKT $\mathrm{SL}(n,\mathbb{H})$-structure on a compact hypercomplex manifold.
More precisely,

\begin{cor}\label{cor:non-existence-hkt}
Let $(M,I)$ be a $4n$-dimensional compact complex manifold such that there exists a non trivial $\del$-Massey product, then $(M,I)$ does not admit any complex structures $J,K$ such that $(M,I,J,K)$ is hypercomplex and admits a HKT $\mathrm{SL}(n,\mathbb{H})$-structure.
\end{cor}

Notice that, in fact, by \cite{barberis-dotti-verbitsky} if a nilmanifold $N$ admits an invariant HKT structure $(I,J,K,\Omega)$ then the complex structures $I,J,K$ are abelian
and in such a case the triple $\partial$-Massey products are trivial. Indeed, we prove in general the following

\begin{theorem}\label{thm:vanishing-massey}
Let $N=\Gamma\backslash G$ be a $2n$-dimensional nilmanifold and let $I$ be an invariant abelian complex structure on $N$. Then, the triple $\partial$-Massey products are all zero.
\end{theorem}

\begin{proof}
Since $I$ is an invariant abelian complex structure on $N$, there exists a co-frame of invariant $(1,0)$-forms $\left\lbrace\varphi^i\right\rbrace_{i=1,\ldots,n}$  on $(N,I)$ such that
$$
\partial\varphi^i=0, \quad\text{for }i=1,\ldots,n\,.
$$
Since $I$ is abelian, by \cite{console-fino} the Dolbeault cohomology of $N$ can be computed using only invariant forms, hence
$$
H^{r,0}_\partial(N)\simeq H^{r,0}_\partial(\g^\mathbb{C})=\left\langle\varphi^{i_1}\wedge\ldots\wedge\varphi^{i_r}\right\rangle_{1\leq i_1<\ldots<i_r\leq n}
$$
for $r=1,\ldots,n$, where, denoting with $\g=\text{Lie}(G)$, $H^{\bullet,\bullet}_\partial(\g^\mathbb{C})$ denotes the cohomology of the differential bigraded algebra $\Lambda^{\bullet,\bullet}(\g^\mathbb{C})^*$ with respect to the operator $\del$.\\
In order to construct a triple $\partial$-Massey product let $\mathfrak{a}=\left[\alpha\right]\in H^{p,0}_\del(N)$,
$\mathfrak{b}=\left[\beta\right]\in H^{q,0}_\del(N)$  such that
$\mathfrak{a} \cup \mathfrak{b}=0\in H^{p+q,0}_\del(N)$, hence $\mathfrak{a} \cup \mathfrak{b}=0\in H^{p+q,0}_\partial(\g^\mathbb{C})$, namely there exists an invariant $(p+q-1,0)$-form $\lambda$ such that
$$
\alpha\wedge\beta=\del\lambda.
$$
But, on invariant $(r,0)$-forms the operator $\partial$ vanishes and so we can take the primitive $\lambda=0$ itself. A similar conclusion is obtained taking the third class in the definition of $\del$-Massey products.
This means that we cannot construct non trivial $\del$-Massey products since both $\lambda$ and $\mu$ in the definition of $\del$-Massey products would be zero.
\end{proof}

An immediate consequence of this result combined with \cite[Theorem 4.6]{barberis-dotti-verbitsky}  is the following
\begin{theorem}\label{thm:vanishing-massey-nilmanifold}
Let $N=\Gamma\backslash G$ be a $4n$-dimensional nilmanifold and let $(I,J,K,\Omega)$ be an invariant HKT structure on $N$. Then, the triple $\partial$-Massey products are all zero.
\end{theorem}

Therefore, a relevant application of Corollary \ref{cor:non-existence-hkt} can be given on solvmanifolds.

\begin{ex}
Consider the $8$-dimensional almost abelian Lie algebra $\g$ with structure equations
$$
[e_8,e_2]=e_4,\quad [e_8,e_3]=e_5.
$$
Let $G$ be the associated solvable simply connected Lie group. Then, by \cite{bock} $G$ admits a lattice $\Gamma$ such that $S:=\Gamma\backslash G$ is a solvmanifold.
Define the complex structure setting as global co-frame of $(1,0)$-forms
$$
\varphi^{1}=e^1+ie^8,\quad
\varphi^{2}=e^2+ie^3,\quad
\varphi^{3}=e^4+ie^5,\quad
\varphi^{4}=e^6+ie^7.
$$
The complex structure equations become
$$
d\varphi^1=d\varphi^2=d\varphi^4=0,\quad
d\varphi^3=\frac{i}{2}\varphi^{12}+\frac{i}{2}\varphi^{2\bar1}.
$$
Note that the form $ \varphi^{1234} $ is closed, so $ S $ has holomorphically trivial canonical bundle.
We now construct a non trivial triple $\del$-Massey product. Take
$[\varphi^1]\in H^{1,0}_\del(S)$,
$[\varphi^2]\in H^{1,0}_\del(S)$ and
$[\varphi^2]\in H^{1,0}_\del(S)$.
Notice that
$\varphi^1\wedge\varphi^2=\del(-2i\,\varphi^3)$ and $\varphi^2\wedge\varphi^2=0$. Hence, the $\del$-Massey product is given by
$$
[-2i\,\varphi^{3}\wedge\varphi^2]\in
 \frac{H^{2,0}_{\del}(S)}
{[\varphi^1]\cup H^{1,0}_{\del}(S)
+H^{1,0}_{\del}(S)\cup [\varphi^2]}.
$$
and this class is clearly non trivial. Therefore, by Corollary \ref{cor:non-existence-hkt} the complex manifold $(S,I)$ does not admit any complex structures $J,K$ such that the solvmanifold $(S,I,J,K)$ is hypercomplex and admits a $\mathrm{SL}(2,\mathbb{H})$ HKT structure. 
\end{ex}

The next two examples show that the converse of Corollary \ref{cor:triple_Mass_prod} (and hence Theorem \ref{thm:triple_Mass_prod}) does not hold in general. The first one is a compact HKT manifold which is not $ \mathrm{SL}(n,\mathbb{H}) $, while the second one is $ \mathrm{SL}(n,\mathbb{H}) $ but does not admit any HKT metric.

\begin{ex}
	Consider $ \mathrm{SU}(3) $ equipped with a homogeneous hypercomplex structure  $ (I,J,K) $ as constructed in \cite{Joyce,SSTV}. There is an HKT metric on $ \mathrm{SU}(3) $ compatible with this hypercomplex structure \cite{Grantcharov-Poon,Opfermann-Papadopoulos}. By \cite{Soldatenkov} the holonomy of the Obata connection on $ \mathrm{SU}(3) $ is $ \mathrm{GL}(2,\mathbb{H}) $ and, in fact, we claim that the $ \partial \partial_J $-lemma cannot hold on $ \mathrm{SU}(3) $.
	
	To see this, we observe that from \cite{GV}, there exists a unitary co-frame $ \{Z^1,\dots,Z^4\} $  of $ (1,0) $-forms (with respect to $ I $) on the Lie algebra of $ \mathrm{SU}(3) $ such that the HKT form is
	\[
	\Omega=Z^{12}+Z^{34}=\frac{1}{2}\partial Z^2.
	\]
	Now, if the $ \partial \partial_J $-lemma hold we would have that $ \Omega=\partial \partial_Jf $ for some function $ f $, but since the HKT form is q-positive, by E. Hopf's maximum principle $ f $ would be constant and thus $ \Omega=0 $ which is a contradiction. 
	
	On the other hand the triple $ \partial $-Massey products are all zero because the same coframe satisfies
	\[
	\partial Z^1=0,\qquad \partial Z^2=2Z^{12}+2Z^{34}, \qquad \partial Z^3=(1+3i)Z^{13},\qquad \partial Z^4=(1-3i)Z^{14},
	\]
	which shows that $ H^{1,0}_\partial (M)\simeq \langle Z^1 \rangle $ and $ H^{i,0}_\partial(M)=0 $ for $ i>1 $.

\end{ex}

\begin{ex}
Consider the nilmanifold $M=\Gamma\backslash G$ whose structure equations of the Lie algebra $\mathfrak{g}$ of $G$ are given by (see \cite[Example 1]{lejmi-weber})
$$
de^1=de^2=de^3=de^4=de^5=0,\quad
de^6=e^{12}+e^{34},\quad
de^7=e^{13}-e^{24},\quad
de^8=e^{14}+e^{23},
$$
where we use the standard notation $e^{ij}=e^i\wedge e^j$.
Define the following hypercomplex structure
$$
Ie^1=e^2,\quad Ie^3=e^4,\quad
Ie^5=e^6,\quad Ie^7=e^8,
$$
$$
Je^1=e^3, \quad Je^2=-e^4, \quad
Je^5=e^7, \quad Je^6=-e^8.
$$
Then a co-frame for invariant $(1,0)$-forms with respect to $ I $ on $M$ is given by
$$
\varphi^{1}=e^1-ie^2,\quad
\varphi^{2}=e^3-ie^4,\quad
\varphi^{3}=e^5-ie^6,\quad
\varphi^{4}=e^7-ie^8
$$
and the complex structure equations become
$$
d\varphi^1=d\varphi^2=0,\quad
d\varphi^3=-\frac{1}{2}(\varphi^{1\bar1}+\varphi^{2\bar2}),\quad
d\varphi^4=\varphi^{12}.
$$
Since the hypercomplex structure is not abelian it does not admit any compatible HKT metric (see also \cite{lejmi-weber}). The conjugate Dolbeault cohomology in bidegree $(p,0)$ is given by
$$
H^{1,0}_{\del}(M)\simeq\left\langle \varphi^{1}, \varphi^{2}, \varphi^{3}\right\rangle\,,\qquad H^{2,0}_{\del}(M)\simeq\left\langle \varphi^{13}, \varphi^{23}, \varphi^{14},\varphi^{24}\right\rangle\,,
$$
We now construct a non trivial triple $\del$-Massey product. Take
$[\varphi^1]\in H^{1,0}_\del(M)$,
$[\varphi^2]\in H^{1,0}_\del(M)$ and
$[\varphi^2]\in H^{1,0}_\del(M)$.
Notice that
$\varphi^1\wedge\varphi^2=\del\varphi^4$ and $\varphi^2\wedge\varphi^2=0$. Hence, the $\del$-Massey product is given by
$$
[\varphi^{4}\wedge\varphi^2]\in
 \frac{H^{2,0}_{\del}(M)}
{[\varphi^1]\cup H^{1,0}_{\del}(M)
+H^{1,0}_{\del}(M)\cup [\varphi^2]}.
$$
and this class is clearly non trivial.
\end{ex}

\section{Balanced HKT solvmanifolds}\label{Sec:solv}

The following result can be seen as a generalization of \cite[Proposition 4.11]{barberis-dotti-verbitsky} where it is proven that for a hyperhermitian nilmanifold with abelian hypercomplex structure the metric is balanced. Notice that in the nilpotent case the HKT assumption is automatic due to \cite{dotti-fino-2}.

\begin{theorem}\label{thm:abelian-then-balanced}
Let $(\Gamma\backslash G,I,J,K)$ be a $4n$-dimensional solvmanifold with an invariant abelian hypercomplex structure. Then, every invariant hyperhermitian metric $g$ is balanced.
\end{theorem}
\begin{proof}
Let $ g $ be an invariant hyperhermitian metric on $ \Gamma \backslash G $, then by \cite{dotti-fino-2} $ g $ is HKT.
We will denote with $(I,J,K,\Omega,g)$ the induced structure on $G$.
Since $\Omega$ is HKT the Bismut connections associated to $I,\,J,\,K$ coincide and we will denote them uniquely with $\nabla^B$.
Since $\Gamma\backslash G$ is a solvmanifold then $G$ is unimodular. Hence, 
by \cite[Lemma 2.4]{barberis-dotti-verbitsky} the Lee form $\tau_J$  associated to $(J,g)$ is given by
$$
\tau_J(X)=\text{tr}\left(\frac{1}{2}J\nabla^B_{JX}\right)
$$
for any $X\in\mathfrak{g}$.
Now we argue as in the proof of \cite[Proposition 4.11]{barberis-dotti-verbitsky} to show that $\tau_J=0$ and so $g$ is balanced with respect to $J$.
The argument is similar for $I$ and $K$.\\
Let $X_1,\,IX_1,\,JX_1,\,KX_1,\,\cdots,\,X_n,\,IX_n,\,JX_n,\,KX_n$ be an orthonormal basis of $\mathfrak{g}$. Now using that $\nabla^B$ preserves $I,\,J,\,K$ and that $g$ is quaternionic Hermitian we have
$$
\begin{aligned}
\text{tr}\left(J\nabla^B_{JX}\right)
&=
\sum_{j=1}^ng(J\nabla^B_{JX}X_j,X_j)+
\sum_{j=1}^ng(J\nabla^B_{JX}IX_j,IX_j)\\
&\quad
+\sum_{j=1}^ng(J\nabla^B_{JX}JX_j,JX_j)+
\sum_{j=1}^ng(J\nabla^B_{JX}KX_j,KX_j)\\
&=\sum_{j=1}^ng(J\nabla^B_{JX}X_j,X_j)+
\sum_{j=1}^ng(JI\nabla^B_{JX}X_j,IX_j)\\
&\quad
+\sum_{j=1}^ng(\nabla^B_{JX}JX_j,X_j)+
\sum_{j=1}^ng(JK\nabla^B_{JX}X_j,KX_j)\\
&=
\sum_{j=1}^ng(J\nabla^B_{JX}X_j,X_j)-
\sum_{j=1}^ng(IJ\nabla^B_{JX}X_j,IX_j)\\
&\quad
+\sum_{j=1}^ng(\nabla^B_{JX}JX_j,X_j)-
\sum_{j=1}^ng(KJ\nabla^B_{JX}X_j,KX_j)\\
&=
\sum_{j=1}^ng(\nabla^B_{JX}JX_j,X_j)-
\sum_{j=1}^ng(\nabla^B_{JX}JX_j,X_j)\\
&\quad
+\sum_{j=1}^ng(\nabla^B_{JX}JX_j,X_j)-
\sum_{j=1}^ng(\nabla^B_{JX}JX_j,X_j)=0.
\end{aligned}
$$
\end{proof}

\begin{cor}
Let $(\Gamma\backslash G,I,J,K)$ be a $4n$-dimensional solvmanifold with an invariant abelian hypercomplex structure. Suppose that there exists an HKT structure $\Omega$ on $(\Gamma\backslash G,I,J,K)$.
Then, there exists a balanced abelian HKT structure on $\Gamma\backslash G$.
\end{cor}
\begin{proof}
By \cite{fino-grantcharov} there exists a invariant HKT structure $\tilde\Omega$ on $(\Gamma\backslash G,I,J,K)$. Now, by Theorem \ref{thm:abelian-then-balanced} we have that $\tilde\Omega$ is balanced.
\end{proof}

\begin{rem}
Notice that, differently from the nilpotent case (cf. \cite{barberis-dotti-verbitsky}), the converse of Theorem \ref{thm:abelian-then-balanced} is not true. Indeed, in \cite{barberis-fino} it is provided an example of a balanced HKT solvmanifold with an hypercomplex structure that is not abelian.
\end{rem}

In \cite{Verbitsky (2007)} Verbitsky showed that an $ \mathrm{SL}(n,\mathbb{H}) $-manifold has holomorphically trivial canonical bundle. Andrada and Tolcachier \cite{Andrada-Tolcachier} found a counterexample to the converse exhibiting an hypercomplex solvmanifold that is not $ \mathrm{SL}(n,\mathbb{H})$ but admits a non-invariant holomorphic section of the canonical bundle. We indeed now show that the $ \mathrm{SL}(n,\mathbb{H}) $ condition is equivalent to the existence of an invariant holomorphic section of the canonical bundle.

\begin{theorem}\label{thm:solv_SLnH}
	Let $(M:=\Gamma\backslash G,I,J,K,g)$ be a hypercomplex solvmanifold, then the holonomy of the Obata connection $ \nabla $ is contained in $ \mathrm{SL}(n,\mathbb{H}) $ if and only if the canonical bundle admits an invariant holomorphic section.
\end{theorem}
\begin{proof}
	Assume $M$ is a $\mathrm{SL}(n,\mathbb{H})$-manifold, then there exists a q-positive holomorphic section $\eta$ of the canonical bundle. We claim that $\eta $ must be invariant. To prove this we argue along the lines of \cite[Proposition 2.1]{fino-otal-ugarte}. Let $\Theta$ be an invariant q-positive section of the canonical bundle (a priori not holomorphic). Since both $\eta$ and $\Theta$ are q-real and q-positive there exists a positive real-valued function $f$ such that $\eta=f\Theta$. Therefore $0=\bar \partial \eta= \bar \partial f \wedge \Theta + f\bar \partial \Theta$, i.e. $\bar \partial \Theta=-\bar \partial(\log f)\wedge \Theta$ because $f$ is positive. Since $\Theta$ is invariant so is $\bar \partial \Theta$ and thus there exists an invariant $(0,1)$-form $\alpha$ such that $\bar \partial(\log f) = \alpha$. We now apply the well-known Belgun's symmetrization process \cite{Belgun} that for any $k$-form $\beta$ on $M$ returns an invariant $k$-form $\mu(\beta)$. Since the hypercomplex structure is invariant $\mu$ preserves $\bar \partial$ and we obtain
	\[
	\bar \partial(\log  f)=\alpha=\mu(\alpha)=\mu(\bar \partial (\log f))=\bar \partial \mu(\log f)=0
	\]
	because the symmetrization of a function is constant. In particular $ f$ is constant, showing that $\eta=f\Theta$ is invariant.
		
	Conversely, let $\eta$ be an invariant nowhere vanishing holomorphic section of the canonical bundle. The fact that $ \mathrm{Hol}(\nabla)\subseteq\mathrm{SL}(n,\mathbb{H}) $, follows from the fact that $ \eta $ is parallel with respect to $ \nabla $, which can be proved along the lines of \cite[Theorem 3.2]{barberis-dotti-verbitsky}.
\end{proof}

Now we prove the following
\begin{theorem}\label{thm:trivialcan-then-balanced}
Let $(M:=\Gamma\backslash G,I,J,K,g)$ be a solvmanifold with an invariant holomorphic section of the canonical bundle and invariant HKT structure. Then $g$ is balanced.
\end{theorem}
\begin{proof}
Let $\bar\eta$ be an invariant non-vanishing $\partial$-closed section of $A^{0,2n}(M)$, hence
$$
\bar\Omega^n=c\,\bar\eta
$$
with $c$ constant. Since $d\bar\eta=0$, then $d\bar\Omega^n=0$ and so $\partial\bar \Omega^n=0$ proving that $g$ is balanced. 
\end{proof}
As a consequence we confirm the conjecture by Alesker and Verbitsky on solvmanifolds with invariant hypercomplex structure.
\begin{theorem}\label{Teor:solv}
Let $(M:=\Gamma\backslash G,I,J,K)$ be a $\mathrm{SL}(n,\mathbb{H})$-solvmanifold with invariant hypercomplex structure. Suppose that there exists an HKT metric on $ M $. Then there exists a balanced HKT structure on $ M $.
\end{theorem}
\begin{proof}
Since $(M:=\Gamma\backslash G,I,J,K,g)$ is an HKT solvmanifold with a $\mathrm{SL}(n,\mathbb{H})$ structure, then the canonical bundle of $M$ has an invariant holomorphic section and, by \cite{fino-grantcharov}, there exists an invariant HKT structure on $M$. Hence, by the previous result the associated Hermitian metric is balanced.
\end{proof}

\end{document}